\documentclass[11pt,reqno]{amsart}
 \usepackage[dvips]{epsfig}
 \usepackage{amsgen, amstext,amsbsy,amsopn, amsthm, amsfonts,amssymb,amscd,amsmat
 h,euscript,enumerate,url,verbatim,calc,xypic}
 \usepackage{hyperref}

 \usepackage{MnSymbol}

 \usepackage{latexsym}
 \usepackage{graphics}
 \usepackage{color}

\usepackage{wrapfig}

 \newcommand{\m}{\mathfrak{m} }

  \newcommand{\Ass}{\operatorname{Ass}}

  \newcommand{\reg}{\operatorname{reg}}

\def\sdefect{\operatorname{sdefect}}

\theoremstyle{plain}
 \newtheorem{theorem}{Theorem}[section]

 \newtheorem{lem}[theorem]{Lemma}
 \newtheorem{proposition}[theorem]{Proposition}

 \newtheorem{thm}[theorem]{Theorem}
 \newtheorem{cor}[theorem]{Corollary}

 \theoremstyle{definition}

 \newtheorem{defn}[theorem]{Definition}
 
 \newtheorem{example}[theorem]{Example}
 \theoremstyle{remark}

\title[Invariants of the symbolic powers of edge ideals] {Invariants of the symbolic powers of edge ideals }
\author[B. Chakraborty and M. Mandal ]{Bidwan Chakraborty$^\dag$ and Mousumi Mandal$^*$}
\thanks{$^\dag$ Supported by CSIR grant No.: $09/081(1303)/2017\mbox{-EMR-I}$, India}
 \thanks{$^*$ Supported by SERB(DST) grant No.: $\mbox{EMR}/2016/006997$, India}

\address{Department of Mathematics, Indian Institute of Technology Kharagpur, 721302, India} \email{bidwan@iitkgp.ac.in}
 \address{Department of Mathematics, Indian Institute of Technology Kharagpur, 721302, India} \email{mousumi@maths.iitkgp.ac.in}
\begin{document}
\maketitle

\begin{abstract}
Let $G$ be a graph and $I=I(G)$ be its edge ideal. When $G$ is the clique sum of two different length odd cycles joined at single vertex then we give an explicit description of the symbolic powers of $I$ and compute the Waldschmidt constant. When $G$ is complete graph then we describe the generators of the symbolic powers of $I$ and compute the Waldschmidt constant and the resurgence of $I$. Moreover for complete graph we prove that  the Castelnuovo-Mumford regularity of the symbolic powers and ordinary powers of the edge ideal coincide.
\end{abstract}

\section{Introduction}
Let $k$ be a field and $R=k[x_1,\ldots ,x_n]$ be a polynomial ring in $n$ variables and $I$ be a homogeneous ideal of $R$. 
Then for $n\geq 1$, the $n$-th symbolic power of $I$ is defined as $I^{(n)}=\displaystyle{\bigcap_{p\in \Ass I}(I^nR_p\cap R)}$. Symbolic powers of an ideal is geometrically an important object  of study as by a classical result of  Zariski and Nagata $n$-th symbolic power of a given ideal consists of the elements that vanish up to order $n$ on the corresponding variety. In general finding the generators of symbolic power is a difficult job. It is easy to see that $I^n\subseteq I^{(n)}$ for all $n\geq 1$. The opposite containment, however, does not hold in general. Much effort has been invested in to determine for which values of $r$ the containment $I^{(r)}\subseteq I^m$ holds. To answer this question C. Bocci and B. Harbourne in \cite{bocci2010} defined an asymptotic quantity known as resurgence which is defined as  $\rho{(I)} =\sup \tiny\{ \frac{s}{t} ~~|~~I^{(s)}\nsubseteq I^t \tiny\}$  and showed that it exists for radical ideals. Since computing the exact value of resurgence is difficult, another asymptotic invariant $ \widehat{\alpha}{(I)} = \displaystyle{\lim_{s\rightarrow\infty}{} \frac{\alpha{(I^{(s)})}}{s}}~,$  known as Waldschmidt constant was introduced, where $\alpha(I)$ denotes the least generating degree of $I$. In order to measure the difference between $I^{(m)}$ and $I^m$, Gattleo et al in  \cite{ggsv} have introduced an invariant known as the symbolic defect which is defined as $ \sdefect(I,m) = \mu\left(\frac{I^{(m)}}{I^m}\right),$ where $\mu(I)$ denotes the minimal number of generators of $I$. It counts the minimal number of generators which must be added to $I^m$ to make $I^{(m)}$.

In this paper we investigate these invariants for edge ideal of graphs. Let $G$ be a simple graph with $n$ vertices $x_1,\ldots ,x_n$ and $I=I(G)$ be the edge ideal generated by $\{x_ix_j|~x_ix_j \mbox{ is an edge of  } G\}$. In \cite{svv} Simis, Vasconcelos and Villarreal have proved that $G$ is a bipartite graph if and only if $I^s=I^{(s)},$ for every $s\in\mathbb{N}$. So it is interesting to study the symbolic powers of edge ideal of non-bipartite graphs. The class of non-bipartite graph which was studied first is odd cycle by Janssen et al in \cite{janssen2017comparing}. They have described the generators of the symbolic powers of the edge ideal of an odd cycle by using the concept of minimal vertex cover and calculated the invariants associated to the symbolic powers of the edge ideal of the same graph. In \cite{gu2018symbolic} Gu et al have extended these results for the unicyclic graph by explicitly computing the generators of the  symbolic powers. Another important invariant in commutative algebra is Castelnuovo-Mumford regularity. Regularity of the edge ideal and their powers has been extensively studied in literature by many researchers while the regularity of the symbolic powers of edge ideals has not been much explored. It has been conjectured by N. C. Minh that for a finite simple graph $G$
$$\reg I(G)^{(s)}=\reg I(G)^s$$
for $s\in \mathbb N$. The conjecture is true for bipartite graph. In \cite{gu2018symbolic} Gu et al have proved the conjecture for odd cycles. Recently in \cite{jayanthan} Jayanthan and Kumar have proved the conjecture for certain class of unicyclic graph and in \cite{fakhari} Seyed Fakhari has solved the conjecture for unicyclic graph.\\

The work of this paper is mainly motivated by the papers \cite{janssen2017comparing} and \cite{gu2018symbolic}. In this paper we study the structure of the symbolic powers of edge ideal of clique sum of two different  odd length cycles joined at a single vertex and prove Minh's conjecture for the class of complete graphs. In section 2, we recall all the definitions and results that will be required for the rest of the paper. Motivated by the concept of vertex weight described in \cite{janssen2017comparing} by Janssen, in section 3, we find the generators of the symbolic powers of edge ideal of clique sum of two different  odd length cycles joined at a single vertex by explicitly choosing a minimal vertex cover. Using the description of the generating set we compute the Waldschmidt constant. In section 4, we describe the generators of the symbolic powers of edge ideal of complete graph and compute the Waldschmidt constant, the resurgence and establish the symbolic defect partially. We close the paper by showing that  for a complete graph $G$ and for all $s\geq 1$,
$$\reg I(G)^{(s)}=\reg I(G)^s.$$

 \vskip 0.3cm
 \section{Preliminaries}
 In this section, we collect the notations and terminologies used in this paper. Throughout the paper, $G$ denotes a finite simple graph over the vertex set $V(G)$ and the edge set $E(G).$ 
 \\

%
%

\begin{defn}
 Let $G$ be a graph.
 \begin{enumerate}
\item A collection of the vertices $W \subseteq V (G)$ is called a vertex cover if for any edge $e\in E(G), W\cap e \neq \phi.$ A vertex cover is called minimal if no proper subset of it is also a vertex cover.
\item The vertex cover number of $G$, denoted by $\tau(G)$, is the smallest size of a minimal vertex cover in G.
\item The graph $G$ is called decomposable if there is a proper partition of its vertices $V(G) = \bigcupdot_{i=1}^{r}{V_i}$ such that $\tau(G) =\sum_{i=1}^{r}{\tau(G[V_i ])}$.
In this case $(G[V_1 ],\dots,G[V_r ])$ is called a decomposition of G. If G is not decomposable then G is said to be indecomposable.

\end{enumerate}

\end{defn}


 \begin{defn}
   Let $V' \subseteq V(G)=\tiny\{x_1,x_2,\ldots,x_n\tiny\}$ be a set of vertices. For a monomial $ x^{\underline{a}} \in k[x_1,\ldots,x_n]$ with exponent vector $\underline{a} = (a_1,a_2,\ldots ,a_n)$ define the vertex weight $W_{V'}{(x^{\underline{a}})}$ to be $$ W_{V'}{(x^{\underline{a}})} := \sum_{x_i\in V'}{a_i}. $$
 \end{defn}

Now we recall some results which  describe the symbolic powers of the edge ideal in terms of minimal vertex covers of the graph.

\begin{lem}\cite[Corollary 3.35]{vantuylbook}
Let $G$ be a graph on vertices $\{x_1,\ldots,x_n\}, I=I(G)\subseteq k[x_1,\ldots,x_n]$ be the edge ideal of $G$ and $V_1,\ldots,V_r$ be the minimal vertex covers of $G.$ Let $P_j$ be the monomial prime ideal generated by the variables in $V_j.$ Then $$I=P_1 \cap\cdots\cap P_r$$ and $$I^{(m)}=P_1^{m} \cap\cdots\cap P_r^{m}.$$
\end{lem}
In the next lemma, the elements in the symbolic power of edge ideals have been described in terms of minimal vertex cover of the graph.

\begin{lem}\cite[Lemma 2.6]{bocci2016}\label{psymbolic}
Let $I \subseteq R$ be a square free monomial ideal with minimal primary decomposition $I=P_1\cap\cdots \cap P_r ~~with ~~P_j = (x_{j_1},\dots ,x_{j_{s_j}})$ for $j= 1,\ldots ,r$. Then $ {x_1^{a_1}}\cdots {x_n^{a_n}} \in I^{(m)} \mbox{ if and only if } a_{j_1}+\dots +a_{j_{s_j}}\geq m$ for $j=1,\dots,r$.

\end{lem}

 Using Lemma \ref{psymbolic} and the concept of vertex weight Janssen et al in  \cite{janssen2017comparing} described  the elements of symbolic powers of edge ideals as follows
$$ I^{(t)} = (\{ x^{\underline{a}} ~|\mbox{ for all minimal vertex covers }V^{\prime}, W_{V^{\prime}}(x^{\underline{a}})\geq t \}).$$
Further they have divided the elements of the symbolic powers of edge ideals into two sets written as  $I^{(t)}=(L(t))+(D(t)),$ where
  $$L(t) =  \{ x^{\underline{a}}~ | \deg(x^{\underline{a}})\geq 2t\mbox{ and for all minimal vertex covers } V^{\prime}, W_{V^{\prime}}(x^{\underline{a}})\geq t \}$$ and
$$D(t) = \{ x^{\underline{a}}~ | \deg(x^{\underline{a}})< 2t\mbox{ and for all minimal vertex covers } V^{\prime}, W_{V^{\prime}}(x^{\underline{a}})\geq t \}.$$
Thus for any graph, if we are able to identify the elements in $L(t)$ and $D(t)$ then we will be able to describe $I^{(t)}$.
\begin{defn}
 Let $G$ be a graph with $n$-vertices and let $v=(v_1 ,\dots,v_n )\in \mathbb{N}^{n} .$ For a vertex $x \in V (G),$ let $N_G (x)=\{y \in V (G)~|~ \{x,y\}\in E(G)\}$ be its neighborhood.
 \begin{enumerate}
\item The duplication of a vertex $x\in V(G)$ in $G$ is the graph obtained from $G$ by adding a new vertex $x^{\prime}$ and all edges $\{x^{\prime} ,y\}$ for $y \in N_G(x).$
\item The parallelization of G with respect to $v$, denoted by $G^v$ , is the graph obtained from $G$ by deleting the vertex $x_i$ if $v_i = 0$, and duplicating $v_i-1$ times the vertex $x_i$ if $v_i \neq 0.$
\end{enumerate}
\end{defn}

A commonly-used method in commutative algebra when investigating (symbolic) powers
of an ideal is to consider its (symbolic) Rees algebra.

\begin{defn} Let $R$ be a ring and $I$ be an ideal of $R.$ The Rees algebra, denoted by $\mathcal{R}(I),$ and the symbolic Rees algebra, denoted by $\mathcal{R}_s(I),$ of $I$ are defined as $$\mathcal{R}(I):= \displaystyle {\bigoplus_{n\geq 0}{I^nt^n} }\subseteq R[t] \mbox{ and } \mathcal{R}_s (I) :=\displaystyle {\bigoplus_{n\geq 0}{I^{(n)}t^n}} \subseteq R[t].$$

\end{defn}

 If $I=I(G)$ is the edge ideal of a graph then the generators of $\mathcal{R}_s(I)$ can be described by indecomposable graphs arising from $G.$ The following characterization for $\mathcal{R}_s(I)$ was given in \cite{bernal}.

 \begin{thm}\label{1implosive}
 Let $G$ be a graph over the vertex set $V(G)=\{x_1,\dots,x_n\}.$ Let $I=I(G)$ be its edge ideal. Then $$\mathcal{R}_s(I)=k[x^vt^b|\mbox{ $G^v$ is an indecomposable graph and } b=\tau(G^v) ],$$ where $v=(v_1,\dots,v_n)$ and $ x^vt^b=x_1^{v_1}\cdots x_n^{v_n}t^b.$
 \end{thm}

 \begin{defn}
Let $G$ be a simple graph of $n$ vertices with edge ideal $I=I(G).$ Then the graph $G$ is called an implosive graph if the symbolic Rees algebra $\mathcal{R}_s(I)$ is generated by monomials of the form $x^vt^b,$ where $v\in\{0,1\}^n.$
 \end{defn}
Next two theorems gives a class of implosive graphs.

\begin{thm}\cite[Theorem 2.3]{flores}\label{pcycle}
If $G$ is a cycle, then $G$ is implosive.
\end{thm}

\begin{defn}
Let $G_1$ and $G_2$ be graphs. Suppose that $G_1\cap G_2=K_r$ is the complete graph of order $r,$ where $G_1\neq K_r$ and $G_2\neq K_r.$ Then, $G_1\cup G_2$ is called the clique-sum of $G_1$ and $G_2.$
\end{defn}

\begin{thm}\cite[Theorem 2.5]{flores}\label{pimplosive}
The clique-sum of implosive graphs is again implosive.

\end{thm}

\section{Symbolic powers of clique sum of two odd cycles joined at a single vertex}
 The work of this section is mainly motived by  \cite[Example 3.5]{gu2018symbolic}. Let $G$ be the clique sum of two different odd length cycles joined at single vertex and $I=I(G)$ be the edge ideal of $G$. In this section we will describe the generators for $I^{(t)}$ by explicitly identifying $(L(t))$ and $(D(t)).$ In \cite[Theorem 4.4]{janssen2017comparing} Janssen et al have proved that if $G$ is an odd cycle and $I$ is the edge ideal of $G$ then $(L(t))=I^t$ for $t\geq 1$. In the following example we show that for the clique sum of two different odd length cycles joined at single vertex $(L(t))\not=I^t$.
\begin{example}
   Let $G$ be the clique-sum of two cycles $C_1=(x_1,x_2,x_3,x_4,x_5)$ and $C_2=(x_1,y_2,y_3,y_4,y_5)$ joined at a single vertex $x_1.$ Consider the monomial $m=x_4y_2y_3y_4y_5x_1$. Clearly $m\notin I^3.$ Also for any vertex cover $V$, $W_V(m)\geq 3$, therefore $m\in(L(3)). $ So $I^3\neq (L(3)).$
\end{example}
 In this case we will try to understand the generators of $(L(t))$. For this we recall the definition of the optimal form introduced in  \cite{janssen2017comparing}.
\begin{defn}
Let $ m \in k[x_1,\ldots,x_{n}]$ be a monomial and $G$ be a finite simple connected graph on the set of vertices $\{x_1,\dots,x_n\}$. Let $\{e_1,e_2,\dots,e_r\}$ denote the set of edges in the graph. We may write $ m=x_1^{a_1}x_2^{a_2} \cdots x_{n}^{a_{n}}e_1^{b_1}e_2^{b_2}\cdots e_{r}^{b_{r}} $,
where $b(m) := \sum{b_j}$ is as large as possible, when $m$ is written in this way, we will call this an optimal form of $m$ or we will say that $m$ is expressed in optimal form, or simply $m$ is in optimal form. In addition, each $ x_i^{a_i} $ with $a_i >0 $ in this form will be called an ancillary factor of the optimal form, or just ancillary for short and $x_i$'s are called ancillary vertices.

\end{defn}
Note that optimal form of a monomial need not be unique but $b(m)$ is unique.

\begin{lem}\label{2optimal}
Let $G$ be a finite simple connected graph on the set of vertices $\{x_1,x_2,\dots,x_n\}$. $I=I(G)$ be the edge ideal. Let $m=x_i e_i^{b_i}\cdots e_{j-1}^{b_{j-1}}\in I(G)$ be a monomial, where $b_k\geq 0$ for $i\leq k \leq j-1$ and $e_i=x_ix_{i+1}.$ Then $mx_j$ will not be an optimal form if and only if the number of vertices within $x_i$ and $x_j$ is even and $b_{i+2h+1}\geq 1$, for $0\leq h\leq \frac{j-i-2}{2}$ with $h\in \mathbb{Z}.$
\end{lem}

\begin{proof}
Since $m$ contains only one ancillary, so $m$ is in optimal form. Here $b(m)=b_i+b_{i+1}+\dots+b_{j-1}.$ Assume that $mx_j$ is not in optimal form, then $b(mx_j)>b(m)$ so there will be no ancillary in $mx_j,$ which implies $x_i$ will no longer be an ancillary, so $x_i$ has to pair up with $x_{i+1}$ to form an edge. The vertex $x_{i+1}$ can come from two edges $e_i$ or $e_{i+1}.$  If $x_{i+1}$ comes from $e_i$, then $x_i$ will be again an ancillary, so that vertex should come from the edge $e_{i+1}$, which implies that the edge $e_{i+1}$ has to be present in $m$. As $x_i$ form an edge with $x_{i+1}$, then   $x_{i+2}$ has to be pair up with some vertex. By similar argument $e_{i+3}$ has to be present in $m$. By repeating this process we get $b_{i+2h+1} \geq 1$ for $0\leq h\leq \frac{j-i-2}{2}.$
As there are no ancillary in $mx_j,$ $x_j$ has to be pair up with $x_{j-1}.$ Then there is only one option to get the edge $x_{j-1}x_j,$ is that $x_{j-1}$ has to come from an edge  $e_{i+2h+1}=x_{i+2h+1}x_{i+2h+2}$ for some $h$, which implies $x_{i+2h+2}=x_{j-1}.$ Thus the number of vertices in $[x_i,x_j]$ is $i+2h+3-i+1=2h+4$ which is even and $b_{i+2h+1}\geq 1$ for $0\leq h\leq \frac{j-i-2}{2}$.

 Conversely if $mx_j=x_i e_i^{b_i}\cdots e_{j-1}^{b_{j-1}}x_j$ with $b_{i+2h+1}\geq 1$ for $0\leq h\leq \frac{j-i-2}{2},$ then by \cite[Lemma 3.4]{janssen2017comparing} $mx_j$ is not in optimal form.
\end{proof}

\begin{defn}
  Let $G$ be the clique sum of two cycles joined at a single vertex. Let $x_i,x_j \in V(G)$ and the vertices between $x_i$ and $x_j$ is even and let $e_i$ denote the edge $x_ix_{i+1}.$ Then the edges $e_{i+1},e_{i+3},\dots,e_{j-2}$ are called alternating edges.
\end{defn}

 In the next lemma we state the key idea of \cite[Theorem 4.4]{janssen2017comparing}, as we will be using this fact in describing the generators of $(L(t)).$

 \begin{lem}\label{cycle}
Let $G$  be an odd cycle and $I=I(G)$ be the edge ideal of $G$. Let $m$ be any monomial then there exists a minimal vertex cover $V$ such that $W_V(m)=b(m)$.
 \end{lem}

\begin{lem}\label{bm}
Let $G$ be the clique sum of two odd cycles  $C_1=(x_1,\ldots,x_{2n+1})$ and $C_2=(x_1,y_2,\ldots,y_{2m+1})$ joined at a vertex $x_1$ with $n\leq  m$ and $I=I(G)$ be the edge ideal of $G$.  Let $\underline{m}$ be a monomial such that  $\underline{m} \notin (c_1,c_2),$ where $c_1=x_1\cdots x_{2n+1}$ and $c_2=x_1y_2\cdots y_{2m+1}$. Then there exists a minimal vertex cover $V$ such that $W_V(\underline m)=b(\underline m)$.
\end{lem}


\begin{center}

\includegraphics[scale=0.55]{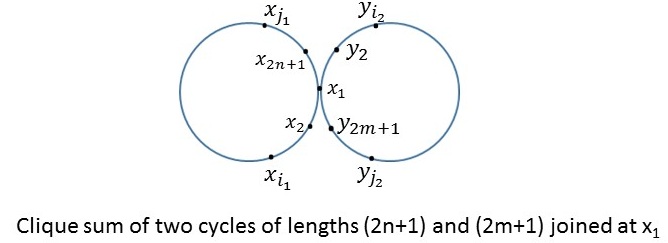}

\end{center}

\begin{proof}
  Let the monomial $\underline{m}$ is of the form
 $\underline{m}=x_{l_1}^{a_{l_1}}x_{l_2}^{a_{l_2}}\cdots x_{l_k}^{a_{l_k}}y_{r_1}^{a_{r_1}}y_{r_2}^{a_{r_2}}\cdots y_{r_p}^{a_{r_p}}e_1^{b_1}\cdots e_{2n+1}^{b_{2n+1}}f_1^{d_1}\cdots\\ f_{2m+1}^{d_{2m+1}},$ where $e_i=x_ix_{i+1}$ for $1\leq i\leq 2n$ and $e_{2n+1}=x_{2n+1}x_1$ and $f_i=y_iy_{i+1}$ for $2\leq i\leq 2m$ and $f_1=x_1y_2$, $f_{2m+1}=y_{2m+1}x_1.$

If $x_1$ is an ancillary vertex of $\underline{m}$ and $\underline{m}\notin(c_1,c_2)$ then by Lemma \ref{cycle} we can choose minimal vertex cover $V_1$ and $V_2$ from the cycles $C_1$ and $C_2$ such that $V=V_1 \cup V_2$ will not contain $x_1$ and $W_V(\underline{m})=b(\underline{m}).$\\ In the rest of the proof we assume that $x_1$ is not an ancillary of $\underline{m}$. Let $[x_i,x_j]$ denote the set of vertices between $x_i$ and $x_j$ including $x_i$ and $x_j$ along clock-wise path and $|[x_i,x_j]|$ denote the number of vertices in that set.  Depending on neighbourhood ancillary vertices of $x_1$ we describe the process to choose the minimal vertex cover $V$ such that $W_V(\underline{m})=b(\underline{m}).$ We have divided the proof into four cases depending on the neighbourhood ancillaries of $x_1$.

 \textbf{Case 1:}  Suppose there are at least two ancillaries from each cycle. Let $x_{i_1}$ and $x_{j_1}$ are two consecutive ancillaries from the cycle $C_1$ and $y_{i_2}$, $y_{j_2}$ are two consecutive ancillaries from the cycle $C_2$ such that $x_1\in [x_{j_1},x_{i_1}]$ and $x_1\in [y_{j_2},y_{i_2}]$, where $i_1<j_1$ and $i_2<j_2$ .
 Let $H_q$ be  the induced subgraph of $G$ on $V_{H_q}=\{x_{j_1},x_{j_1+1},\ldots,x_1,\ldots,x_{i_1},y_{j_2},\ldots,y_{2m+1},y_2,\ldots,y_{i_2} \}.$
  And let $\underline{m}_q=x_{j_1}^{a_{j_1}}e_{j_1}^{b_{j_1}}\cdots e_1^{b_1}\cdots e_{i_1-1}^{b_{i_1-1}}x_{i_1}^{a_{i_1}}y_{j_2}^{c_{j_2}}f_{j_2}^{d_{j_2}}\cdots f_{2m+1}^{d_{2m+1}}f_1^{d_1}\cdots f_{i_2-1}^{d_{i_2-1}}y_{i_2}^{c_{i_2}}.$ 

 Let $\underline{m}_{1_{C_1}}= x_{j_1}^{a_{j_1}}e_{j_1}^{b_{j_1}}\cdots e_{2n+1}^{b_{2n+1}}$, $\underline{m}_{2_{C_1}}= e_1^{b_1}\cdots e_{i_1-1}^{b_{i_1-1}}x_{i_1}^{a_{i_1}}$,
 $\underline{m}_{1_{C_2}}= f_1^{d_1}\cdots f_{i_2-1}^{d_{i_2-1}}y_{i_2}^{c_{i_2}}$ and
 $\underline{m}_{2_{C_2}}= y_{j_2}^{c_{j_2}}f_{j_2}^{d_{j_2}}\cdots f_{2m+1}^{d_{2m+1}}.$

 If $\underline{m}_{1_{C_1}}x_1$, $\underline{m}_{2_{C_1}}x_1$, $\underline{m}_{1_{C_2}}x_1$ and $\underline{m}_{2_{C_2}}x_1$ are in optimal form then $mx_1$ is in optimal form, which means $x_1$ is an ancillary vertex of $mx_1.$ So we can find a minimal vertex cover $V$ such that $W_V(\underline{m}x_1)=b(\underline{m}x_1),$ which implies $W_V(\underline{m})=b(\underline{m}).$  Let us assume that $\underline{m}_{1_{C_1}}x_1$ is not in optimal form.
Then by Lemma \ref{2optimal}, $|[x_{j_1},x_1]|$ is even and $b_{j_1+2h+1}\geq 1$ for $0\leq h\leq \frac{2n-j_1}{2}.$
Now choose the minimal vertex cover for the part $\underline{m}_{1_{C_1}}$ as $C_{1_{V_1}}=\{x_{j_1+1},x_{j_1+3},\ldots,x_1  \}.$
\newline (a)   If $\underline{m}_{1_{C_2}}x_1$ is in optimal form then one of the following condition will hold.
\begin{enumerate}
  \item[(i)] $|[x_1,y_{i_2}]|$ is odd. Then $|[x_{j_1},y_{i_2}]|$ is even and one of the alternating edge between $x_{j_1}$ and $y_{i_2}$ will be missing (as $\underline{m}_{1_{C_1}}  \underline{m}_{1_{C_2}}$ is in optimal form). Since all the alternating edges in $\underline{m}_{1_{C_1}}$ are present so the alternating edge will be missing from the part $\underline{m}_{1_{C_2}}.$
   Therefore $d_{2h+1}=0$  for some $h$, where $0\leq h\leq \frac{i_2-2}{2}.$ Then choose the minimal vertex cover for the part $\underline{m}_{1_{C_2}}$ as $C_{2_{V_1}}= \{  x_1,y_3,\ldots,y_{2h+1},y_{2h+2},y_{2h+4},\ldots,\\y_{i_2-1}\}.$
  \item[(ii)] $|[x_1,y_{i_2}]|$ is even and one of the alternating edges will be missing. Then choose the minimal vertex cover as  $C_{2_{V_1}}= \{ x_1,y_3,\ldots,y_{i_2-1}\}$ (as $i_2-1$ is odd).
\end{enumerate}

(b) If $\underline{m}_{1_{C_2}}x_1$ is not optimal form then $|[x_1,y_{i_2}]|$ is even, so choose the minimal vertex cover as $ C_{2_{V_1}}=\{x_1,y_3,\ldots,y_{i_2-1} \}.$

Similarly we do the same thing for $\underline{m}_{2_{C_2}}x_1.$ If $\underline{m}_{2_{C_2}}x_1$ is in optimal form and if $|[y_{j_2},x_1]|$ is odd then we can deduce from the previous discussion that $d_{j_2+2h+1}=0$ for some $h$, where $0\leq h \leq \frac{2m-j_2}{2}.$ Then take the minimal vertex cover as $C_{2_{V_2}}=\{y_{j_2+1},y_{j_2+3},\ldots,y_{j_2+2h+1},\\y_{j_2+2h+2},y_{j_2+2h+4}, \ldots,x_1\}.$
In other cases $|[y_{j_2},x_1]|$ will be even then choose the minimal vertex cover as
$C_{2_{V_2}}= \{y_{j_2+1},y_{j_2+3},\ldots,x_1\}$.

Again we will find the minimal vertex cover for the part $\underline{m}_{2_{C_1}}.$ If $|[x_1,x_{i_1}]|$ is odd in clock-wise direction then $b_{2h+1}=0$ for some $h$, where $0\leq 2h\leq i_1-2.$ Then take the minimal vertex cover as $C_{1_{V_2}}=\{ x_1,x_3,\ldots,x_{2h+1},x_{2h+2},x_{2h+4},\ldots,x_{i_1-1}\}$ (as $i_1-1$ is even).
In other cases $|[x_1,x_{i_1}]|$ is even, then take the minimal vertex cover as $C_{1_{V_2}}=\{x_1,x_3,\ldots,x_{i_1-1}\}$ (as $i_1-1$ is odd).
%


Our aim is here to find a minimal vertex cover $V$ such that $W_V(\underline{m}_q)=b(\underline{m}_q).$ Depending on whether the four monomial part $\underline{m}_{1_{C_1}}x_1,\underline{m}_{2_{C_1}}x_1,\underline{m}_{1_{C_2}}x_1,\underline{m}_{2_{C_2}}x_1$ are in optimal form or not we can choose the minimal vertex cover. Here we will show for one case and the other cases will be similar.
Assume that $\underline{m}_{1_{C_1}}x_1,$  $\underline{m}_{2_{C_2}}x_1$, $\underline{m}_{2_{C_1}}x_1$ are not in optimal form and $\underline{m}_{1_{C_2}}x_1$ is in optimal form. Then depending on $|[x_1,y_{i_2}]|$ is even or odd there are two possibilities of minimal vertex cover. If $|[x_1,y_{i_2}]|$ is odd then choose the minimal vertex cover for $\underline{m}_q$ as  $V=\{x_{j_1+1},x_{j_1+3},\ldots,x_1,y_3,\ldots,y_{2h+1},y_{2h+2},y_{2h+4},y_{i_2-1},y_{j_2+1}
,y_{j_2+3},\ldots,\\y_{2m}, x_1,x_3,\ldots,x_{i_1-1}\},$ where $d_{2h+1}=0.$
Now $b(\underline{m}_q)=b_{j_1}+\cdots+b_{2n+1}+b_1+\cdots+b_{i_1-1}+d_{j_2}+\cdots+d_{2m+1}+d_1+\cdots+d_{i_2-1}$ and $W_V(\underline{m}_q)=b_{j_1}+\cdots+b_{2n+1}+b_1+d_1+d_{2m+1}+d_2+\cdots+d_{2h}+2d_{2h+1}+d_{2h+2}+\ldots+d_{i_2-1}
+d_{j_2}+\cdots+d_{2m}+b_2+b_3+\cdots+b_{i_1-1}.$ As $d_{2h+1}=0$, $W_V(\underline{m}_q)=b(\underline{m}_q).$

If $|[x_1,y_{i_2}]|$ is even then choose the minimal vertex cover for $\underline{m}_q$ as $V=\{x_{j_1+1},x_{j_1+3},\ldots,x_1,\\y_3,\ldots,y_{i_2-1},y_{j_2+1},y_{j_2+3},\ldots,
y_{2m},x_1,x_3,\ldots,x_{i_1-1}\} $ and note that $W_V(\underline{m}_q)=b(\underline{m}_q).$  Similarly we can write the minimal vertex cover for the remaining cases such that $W_V(\underline{m}_q)=b(\underline{m}_q).$ Now for the vertices of $G$ that are not in $V_{H_q}$, we can choose a minimal vertex cover $V^{\prime}$ by Lemma \ref{cycle}  such that $W_{V^{\prime}\cup V}(\underline{m})=b(\underline{m}).$
\\

\textbf{Case 2:}  Suppose that each cycle contains single ancillary. Let the ancillaries are $x_{j_1}$ and $y_{j_2}$ in $C_1$ and $C_2$ respectively and  $\underline{m}_{1_{C_1}}= x_{j_1}^{a_{j_1}}e_{j_1}^{b_{j_1}}\cdots e_{2n+1}^{b_{2n+1}}$, $\underline{m}_{2_{C_1}}= e_1^{b_1}\cdots e_{i_1-1}^{b_{i_1-1}}x_{j_1}^{a_{i_1}}$,
 $\underline{m}_{1_{C_2}}= f_1^{d_1}\cdots f_{i_2-1}^{d_{i_2-1}}y_{j_2}^{c_{i_2}}$ and
 $\underline{m}_{2_{C_2}}= y_{j_2}^{c_{j_2}}f_{j_2}^{d_{j_2}}\cdots f_{2m+1}^{d_{2m+1}}.$
 Let us assume that
$\underline{m}_{1_{C_1}}x_1$ is not in optimal form. Then $|[x_{j_1},x_1]|$ is even and all the alternating edges will be present. Then choose the minimal vertex cover for $\underline{m}_{1_{C_1}}$as $C_{1_{V_1}}=\{ x_{j_1+1},x_{j_1+3},\ldots,x_1\}.$ As $\underline{m}\notin (c_1,c_2)$, therefore at least one vertex will be missing from each of the cycle in $\underline{m}$. The missing vertex in $C_1$ will be from the part $\underline{m}_{2_{C_1}}$ and will be of the form $x_{2k}$ or $x_{2k+1}.$ If $x_{2k}$ is missing then take the minimal vertex cover for $\underline{m}_{2_{C_1}}$ as (here $b_{2k-1}=0$ and $b_{2k}=0$) $C_{1_{V_2}}= \{ x_1,x_3,\ldots,x_{2k-1},x_{2k},x_{2k+2},\ldots,x_{j_1-1} \}$ (as $j_1-1$ is even).
If $x_{2k+1}$ is missing then take the minimal vertex cover as $C_{1_{V_2}}=\{
x_1,x_3,\ldots,x_{2k+1},x_{2k+2},x_{2k+4},\ldots,x_{j_1-1} \}.$ Now take
$V_1=C_{1_{V_1}}\cup C_{1_{V_2}}.$ Then $V_1$ will be either $\{x_{j_1+1},x_{j_1+3},\ldots,x_1,x_3,\ldots,x_{2k-1},x_{2k},x_{2k+2},\\ \ldots,x_{j_1-1}\}$
 or $\{x_{j_1+1},x_{j_1+3},\ldots,x_1,x_3, \ldots,x_{2k+1},x_{2k+2},x_{2k+4},\ldots,x_{j_1-1}\}.$ Now if $\underline{m}_{1_{C_2}}x_1$  is not in optimal form then  similarly we can construct a minimal vertex cover $V_2$ containing $x_1$ for $C_2$.
Next assume that $\underline{m}_{1_{C_2}}x_1$ is in optimal form. Then there are two possibilities:\\
\romannumeral 1) $|[x_1,y_{j_2}]|$ is odd. In this case we can choose the minimal vertex cover as $ C_{2_{V_1}}=\{x_1,y_3,\ldots,y_{2h+1},y_{2h+2},y_{2h+4},\ldots,y_{j_2-1}\}$ by the method  described in the Case 1. Since $|[y_{j_2},x_1]|$ is even so we can choose the minimal vertex cover for $\underline{m}_{2_{C_2}}$ as $C_{2_{V_2}}=\{y_{j_2+1},y_{j_2+3},\\ \dots ,y_{2m},x_1\}$. \\

\romannumeral 2) $|[x_1,y_{j_2}]|$ is even and one of the alternating edge is missing. In this case we can choose the minimal vertex cover as  $ C_{2_{V_1}}=\{x_1,y_3,\ldots,y_{j_2-1}\}.$  Since $|[y_{j_2},x_1]|$ is odd so we can choose the minimal vertex cover for $\underline{m}_{2_{C_2}}$ as $C_{2_{V_2}}=\{y_{j_2+1},y_{j_2+3},\dots ,y_{j_2+2h+1},y_{j_2+2h+2},\\y_{j_2+2h+4},\dots ,x_1\}$ as described in Case 1.
In both the cases consider $V_2=C_{2_{V_1}}\cup C_{2_{V_2}}$ and $V=V_1\cup V_2.$
Then by similar type of calculation observe that $W_V(\underline{m})=b(\underline{m})$.
\\

%
%
%
%
%
%
%
%

  \textbf{Case 3:} Suppose one cycle contains single ancillary and another cycle contains more than one ancillary, then by Case 1 and Case 2 we can choose the minimal vertex cover $V$ such that $W_V(\underline{m})=b(\underline{m}).$\\

 \textbf{Case 4:} Suppose that all the ancillaries are from one cycle and assume that this cycle is $C_1$. It may contain one ancillary or more than one ancillary. Let us do here for the case more than one ancillary, other cases will be similar. Let us assume that  $x_{j_1}$ and $x_{i_1}$ are two consecutive ancillaries from the cycle $C_1$ such that $x_1\in[x_{j_1},x_{i_1}],$ where $i_1<j_1$ and $V_{H_q}=\{x_{j_1},x_{j_1+1},\dots,x_1,\dots,x_{i_1},y_2,y_3,\dots,y_{2m+1}\}.$
  By Case 1 choose the minimal vertex cover for the monomial parts $\underline{m}_{1_{C_1}}$ and $\underline{m}_{2_{C_1}}.$ Let $V_1$ be the minimal vertex cover for these two parts. Now we will choose the minimal vertex cover for the cycle $C_2.$  Since $\underline{m}\notin (c_1,c_2)$ so at least one vertex  will be missing from the cycle $C_2.$

(a) If the vertex $x_1$ is missing from the monomial then we can keep $x_1$ in the minimal vertex cover and take $V_2=\{x_1,y_2,y_4,\dots,y_{2n}\}$ and $V=V_1\cup V_2.$\\
(b) If $x_1$ is not missing from the monomial then we need to find $V_2.$ Here some $y_r$ will be missing from the monomial. Now $V_1$ may or may not contain $x_1.$ First we consider the case where $V_1$ contains $x_1.$
If the missing vertex is of the form $y_{2k}$ then take the minimal vertex cover as $V_2=\{x_1,y_3,\ldots,y_{2k-1},y_{2k},y_{2k+2},\ldots,y_{2m}\},$ or if missing vertex is of the form $y_{2k+1}$ then take the minimal vertex cover as $V_2=\{x_1,y_3,\ldots,y_{2k+1},y_{2k+2},y_{2k+4},\ldots,\\y_{2m}\}.$  Now if $x_1$ is not missing as well as $x_1$ is not in $V_1$ then take $V_2$ as $\{y_2,y_4,\dots,y_{2k},y_{2k+1},\\\dots,y_{2m}\}.$ Then consider  $V=V_1\cup V_2$ and by similar type of calculation observe that $W_V(\underline{m}_q)=b(\underline{m}_q).$ Now for the vertices of $G$ that are not in $V_{H_q}$, we can choose a minimal vertex cover $V^{\prime}$ by Lemma \ref{cycle}  such that $W_{V^{\prime}\cup V}(\underline{m})=b(\underline{m}).$

\end{proof}

\begin{proposition}\label{lt}
Let $G$ be the clique sum of two odd cycles  $C_1=(x_1,\ldots,x_{2n+1})$ and $C_2=(x_1,y_2,\ldots,y_{2m+1})$ joined at a vertex $x_1$ with $n\leq m$ and $I=I(G)$ be the edge ideal of $G$. Let $c_1=x_1\cdots x_{2n+1}$ and  $c_2=x_1y_2\cdots y_{2m+1}.$ Then $(L(t))\subseteq I^{t} + (c_1,c_2).$


\end{proposition}

\begin{proof} We prove the proposition by contrapositive. Let $\underline{m}$ be a monomial such that $\underline{m}$ $\notin I^{t}$ and $\underline{m} \notin (c_1,c_2).$ By Lemma \ref{bm}, there exist a minimal vertex cover $V$ such that $W_V({\underline{m}})=b({\underline{m}})$. As $\underline{m}\notin I^{t}$, therefore $b(\underline{m})< t.$ As $b(\underline{m})<t$, so $\underline{m}\notin L(t)$.  Since $L(t)$ is a generating set of $(L(t)),$ this is sufficient to claim that $\underline{m}\notin (L(t))$ because neither $\underline{m}$, nor any of its divisors whose vertex weights can only be less than that of $\underline{m}$, will be in the generating set.
\end{proof}


In order to describe $I^{(t)}$ we need the description of $(D(t)).$ Next lemma describes the elements of $(D(t)).$
\begin{lem}\label{dt}
Let $G$ be the clique sum of two odd cycles  $C_1=(x_1,\ldots,x_{2n+1})$ and $C_2=(x_1,y_2,\ldots,y_{2m+1})$ joined at a vertex $x_1$ with $n\leq m$.  Let $I=I(G)$ be the edge ideal of $G$. Then for $t\geq 2$ {\small $$(D(t))= \langle\{I^i(c_1)^s(c_2)^b \mid {i+(n+1)s+(m+1)b \geq t \mbox{ and } 2i+(2n+1)s+(2m+1)b \leq 2t-1 }  \}\rangle $$}
  where $c_1=x_1x_2\cdots x_{2n+1}$ and $c_2=x_1y_2\cdots y_{2m+1}$.

\end{lem}

\begin{proof}
Consider the set\\ $T=\{I^i(c_1)^s(c_2)^b \mid {i+(n+1)s+(m+1)b \geq t \mbox{ and } 2i+(2n+1)s+(2m+1)b \leq 2t-1 } \}.$
From the definition of $D(t)$ it is clear that $T \subseteq D(t).$
Let $x^{\underline{a}}$ be a monomial in $D(t)$. Write $x^{\underline{a}}=(c_1)^s(c_2)^b x^{\underline{d}}$ such that $s$ and $b$ are as large as possible. Now $x^{\underline{d}}$ can be written in the form $E(x^{\underline{d}})A(x^{\underline{d}}),$ where $E(x^{\underline{d}})=  e_{i_1}^{b_{i_1}}\cdots e_{i_k}^{b_{i_k}}f_{j_1}^{b_{j_1}}\cdots f_{j_k}^{b_{j_k}}$ with $b_{i_1}+\dots+b_{i_k}+b_{j_1}+\dots+b_{j_k}=i,$ and $A(x^{\underline{d}})$ is a product of ancillaries of $x^{\underline{d}}$.

 Then from Lemma \ref{bm} we can find a minimal vertex cover $V$ such that $W_V(x^{\underline{d}})=W_V(E(x^{\underline{d}}))=i,$ which implies that $W_V(x^{\underline{a}})=(n+1)s+(m+1)b+i=W_V(E(x^{\underline{d}})(c_1)^s(c_2)^b)\geq t.$ As for any other minimal vertex cover $V^{\prime}$ we have $W_{V^{\prime}}(E(x^{\underline{d}})(c_1)^s(c_2)^b)\geq i+(n+1)s+(m+1)b\geq t$, therefore $E(x^{\underline{d}})(c_1)^s(c_2)^b \in D(t).$ Hence $\langle T\rangle=\langle (D(t))\rangle.$


\end{proof}

%
%

\begin{cor}
$(D(n+1))=(c_1).\label{3c1}$
\end{cor}
\begin{proof}
  From Lemma \ref{dt} it follows that we need to find  all the possible values of $(i,s,b)$ such that  $i+(n+1)s+(m+1)b \geq n+1$ and $2i+(2n+1)s+(2m+1)b \leq 2n+1.$ Then it is clear that $(0,1,0)$ is the only possible value of $(i,s,b)$. Therefore $D(n+1)$ will contain only $c_1.$
\end{proof}

\begin{cor}\label{t3}
  If $C_1$ and $C_2$ are of same length say $2n+1$ then $(D(n+1))=(c_1,c_2).$
\end{cor}
\begin{proof}
  In Lemma \ref{dt} put $m=n$, then we need to find all the possible values of $(i,s,b)$ such that  $i+(n+1)(s+b) \geq n+1$  and  $2i+(2n+1)(s+b) \leq 2n+1.$ Observe that $(0,1,0),(0,0,1)$ are the only possible values of $(i,s,b)$ . Thus $D(n+1)=\{c_1,c_2\}.$
\end{proof}

All the generators of $D(t)$ described in Lemma \ref{dt} need not be minimal. In the next lemma we identify the minimal generators of $D(t)$ for $n+1\leq t\leq m+1$.

\begin{cor}\label{3dmin}
  The minimal set of generators of the ideal $(D(t))$ for $n+1\leq t \leq m+1$ is given by the following set $$D_{\min}(t)=\{I^i{(c_1)}^s{(c_2)}^b \mid i+(n+1)s+(m+1)b=t \mbox{ and } i<t\}.$$
\end{cor}
\begin{proof}
If $b\geq 1$ then by Corollary \ref{dt} there is only one possible generator namely $c_2$  corresponding to $(i,s,b)=(0,0,1)$, which is in fact a minimal generator for $(D(t))$ and this is only for $t=m+1.$ To find the remaining minimal generators let us  assume that $b=0.$ Then it is clear that $I^ic_1^s$ with $i+(n+1)s=t$ are minimal generators.
Let us now assume that $I^ic_1^s$ with $i+(n+1)s>t$ is a minimal generator.  But since $(i-1)+(n+1)s\geq t$ therefore $I^{i-1}c_1^s\in (D(t))$ which will contradict that $I^ic_1^s$ is a minimal generator of $(D(t))$. Hence the proof.
\end{proof}

\begin{thm}

  Let $G$ be the clique-sum of two odd cycles $C_1=(x_1,\ldots,x_{2n+1})$ and $C_2=(x_1,y_2,\ldots,y_{2m+1})$ joined at a vertex $x_1$ with $n\leq m.$ Let $I=I(G)$ be the edge ideal. Then for any $ t\in \mathbb{N}$, $ \alpha{(I^{(t)})}$
     =$ 2t- \Bigl\lfloor\dfrac{t}{n+1}\Bigr\rfloor.\qquad$\\
     Particularly, the Waldschmidt constant of $I$ is given by \begin{align*}\widehat{\alpha}{(I)}=\frac{2n+1}{n+1}.\end{align*}

\end{thm}

\begin{proof}Since the degree of the elements of $D(t)$ is strictly less than the degree of the elements of $L(t),$ so if $D(t)$ is nonempty then minimal generating degree term will come from $D(t).$ Note that $\alpha(I^{(t)})=\min\{2i+(2n+1)s+(2m+1)b \mid i+(n+1)s+(m+1)b=t\}.$ Let $d=2i+(2n+1)s+(2m+1)b=2t-(s+b).$ Then $d$ will be minimum when $s+b$ will be maximum. Write $t=k(n+1)+r,$ where $0\leq r\leq n.$ In general all the possible choices of $(i,s,b)$ are $(r_1,k_1,k_2),$ where $t=k_1(n+1)+k_2(m+1)+r_1.$ We will prove that $k\geq k_1+k_2.$ Write $k_2(m+1)=k_3(n+1)+r_2,$ where $r_2\leq n.$ Therefore $k_3\geq k_2.$ So $t=(k_1+k_3)(n+1)+(r_1+r_2).$ Again write $r_1+r_2=k_4(n+1)+r_4,$ where $r_4\leq n.$ Then $t=(k_1+k_2+k_3)(n+1)+r_4,$ with $r_4\leq n$, which implies that $k=k_1+k_3+k_4$ and $r=r_4.$ Hence $k\geq k_1+k_3 \geq k_1+k_2.$ Thus $\alpha(I^{(t)})=2t-k=2t- \Bigl\lfloor\dfrac{t}{n+1}\Bigr\rfloor.$ Therefore $\widehat{\alpha}{(I)} = 2- \frac{1}{n+1}=\frac{2n+1}{n+1}$.

\end{proof}

\begin{lem}\label{ld}
Let $G$ be the clique sum of two odd cycles  $C_1=(x_1,\ldots,x_{2n+1})$ and $C_2=(x_1,y_2,\ldots,y_{2m+1})$ joined at a vertex $x_1$ with $n<m$.  Let $I=I(G)$ be the edge ideal of $G$ and $c_1=x_1x_2\cdots x_{2n+1},$ $c_2=x_1y_2\cdots y_{2m+1}.$ Let $x^{\underline{a}}\in I^{(t)}$ be a minimal generator of $I^{(t)}$ such that  $x^{\underline{a}}\in (c_1)$ and $x^{\underline{a}}\notin I^{t}$ then $x^{\underline{a}}\in (D(t))$ for $n+1\leq t\leq m+1.$
\end{lem}
\begin{proof}Write $x^{\underline{a}}=c_1^sx^{\underline{d}}$ such that $s$ is as large as possible. Now $x^{\underline{d}}$ can be written in the form $E(x^{\underline{d}})A(x^{\underline{d}}),$ where $E(x^{\underline{d}})=  e_{i_1}^{b_{i_1}}\cdots e_{i_k}^{b_{i_k}}f_{j_1}^{b_{j_1}}\cdots f_{j_k}^{b_{j_k}}$ with $b_{i_1}+\dots+b_{i_k}+b_{j_1}+\dots+b_{j_k}=i$ and $A(x^{\underline{d}})$ is a product of ancillaries of $x^{\underline{d}}$. By Proposition \ref{lt}, we can choose a minimal vertex cover $V^{\prime}$ excluding the vertices that appear in ancillaries such that $W_{V^{\prime}}(x^{\underline{d}})=W_{V^{\prime}}(E(x^{\underline{d}}))=i, \text { which implies that } W_{V^{\prime}}(x^{\underline{a}})=W_{V^{\prime}}(c_1^sE(x^{\underline{d}}))=(n+1)s+i\geq t.$
 As for any minimal vertex cover $V$, $W_V(c_1^sE(x^{\underline{d}}))\geq W_{V^{\prime}}(c_1^sE(x^{\underline{d}})),$ so $c_1^sE(x^{\underline{d}})\in I^{(t)}.$ Since we are interested in minimal generators so we can take the monomial as $x^{\underline{a}}=c_1^sE(x^{\underline{d}}).$ Then $\deg(x^{\underline{a}})=(2n+1)s+2i.$ Now from Corollary \ref{3dmin} it follows that $(2n+1)s+2i=2[(n+1)s+i]-s=2t-s<2t.$
\end{proof}

\begin{thm}\label{3differentlength}
Let $G$ be the clique sum of two odd cycles  $C_1=(x_1,\ldots,x_{2n+1})$ and $C_2=(x_1,y_2,\ldots,y_{2m+1})$ joined at a vertex $x_1$ with $n<m$.  Let $I=I(G)$ be the edge ideal of $G$. Let $c_1=x_1\cdots x_{2n+1}$ and $c_2=x_1y_2\cdots y_{2m+1}$. Then
\begin{enumerate}
  \item For $1\leq t\leq n,$ we have $I^{(t)}=I^t.$
  \item $I^{(n+1)}= I^{n+1}+(c_1).$
  \item For $n+2\leq t\leq m+1$ we have $I^{(t)}=I^{t}+(D(t)).$
\end{enumerate}

\end{thm}

\begin{proof}
\begin{enumerate}
\item The proof follows from \cite[Corollary 4.5]{lam2015associated}.
\item Since $c_2\in I^m$ and $n+1\leq m$ thus $c_2\in I^{n+1}$. Then by Proposition \ref{lt},  $L(n+1)\subseteq I^{n+1}+c_1.$ Thus by Corollary \ref{3c1}, $I^{(n+1)}\subseteq I^{n+1}+c_1.$ Hence $I^{(n+1)}=I^{n+1}+c_1.$
\item From Lemma \ref{ld} and Proposition \ref{lt}, it follows that $(L(t))\subseteq I^{t}+(D(t)).$ So $I^{(t)}\subseteq I^{t}+(D(t)).$ Hence $I^{(t)}=I^{t}+(D(t)).$
\end{enumerate}
\end{proof}

\begin{cor}\label{3samelength}
Let $G$ be the clique sum of two odd cycles  $C_1=(x_1,\ldots,x_{2n+1})$ and $C_2=(x_1,y_2,\ldots,y_{2n+1})$ of same length joined at a vertex $x_1$.  Let $I=I(G)$ be the edge ideal of $G$ and $c_1=x_1\cdots x_{2n+1}$ and $c_2=x_1y_2\cdots y_{2n+1}$. Then
 \begin{enumerate}
                \item $I^{(s)}=I^s$ for $1\leq s\leq n$.
                \item $I^{(n+1)}= I^{n+1} + (x_1\cdots x_{2n+1}) + (x_1y_2\cdots y_{2n+1}).$
              \end{enumerate}
\end{cor}

\begin{proof}
  \begin{enumerate}
    \item The proof follows from \cite[Corollary 4.5]{lam2015associated}.
    \item The proof follows from Proposition \ref{lt} and Corollary \ref{t3}.
  \end{enumerate}
\end{proof}

In \cite[Lemma 3.1]{jayanthan} Jayanthan and Kumar have described the structure of the symbolic Rees algebra for the clique sum of same length odd cycles and computed the invariants. In the next theorem we describe the structure of the symbolic Rees algebra for the clique sum of two different length odd cycles joined at a single vertex using the description of $L(t)$ and $D(t)$.

\begin{thm}\label{t2}
\begin{enumerate}
\item

  Let $G$ be clique sum of two same length odd cycles $C_1=(x_1,\ldots,\\x_{2n+1})$ and $C_2=(x_1,y_2,\ldots,y_{2n+1})$ joined at a single vertex $x_1$. Let $I=I(G)$ be the edge ideal. Let $s\in \mathbb{N}$ and write $s=k(n+1)+r$ for some $k \in\mathbb{Z} $ and $0\leq r\leq n.$ Then $$I^{(s)}=\displaystyle  {\sum_{p+q=t=0}^{k}{I^{s-t(n+1)}(x_1\cdots x_{2n+1})^p(x_1y_2\cdots y_{2n+1})^q}}.$$

\item
\label{t1}
Let $G$ be the clique sum of two odd cycles  $C_1=(x_1,\ldots,x_{2n+1})$ and $C_2=(x_1,y_2,\ldots,\\y_{2m+1})$ joined at a vertex $x_1$ with $n<m$.  Let $I=I(G)$ be the edge ideal of $G$ and $c_1=x_1\cdots x_{2n+1}.$ Let $s=k_1(n+1)+r_1$, where $0\leq r_1\leq n$ and $s=k_2(m+1)+r_2$, where $0\leq r_2\leq m.$ Then\\
$$I^{(s)}=\displaystyle{\sum_{\substack{t_1,t_2\\0\leq t_1\leq k_1~and \\0\leq t_2\leq k_2 ~and \\ s-t_1(n+1)-t_2(m+1)\geq 0}}{I^{s-t_1(n+1)-t_2(m+1)}(c_1)^{t_1}(D_{\min}(m+1))^{t_2}}}.$$

\end{enumerate}
\end{thm}

\begin{proof} By Theorems \ref{pcycle} and \ref{pimplosive}, $G$ is an implosive graph in either case. Thus symbolic Rees algebra $\mathcal{R}_s(I)$ of $I$ is generated by the monomials of the form $x^vt^b,$ where $v\in \{0,1\}^{\mid V(G)\mid}$ and $G^v$ is induced indecomposable subgraph of $G.$ It can be seen from \cite[Corollary 2a]{Frank} that induced indecomposable subgraphs of $G$ is either an edge or an odd cycle.
\begin{enumerate}
\item
  Thus symbolic Rees algebra is minimally generated only in degrees 1 and $n+1$, so we have
  \begin{align*} I^{(s)} &= \displaystyle{\sum_{r+k(n+1)=s}{I^r{(I^{(n+1)})}^k}}\\ &=\displaystyle {\sum_{p+q=t=0}^{k}{I^{s-t(n+1)}
  (x_1\cdots x_{2n+1})^p(x_1y_2\cdots y_{2n+1})}^q}\text{ (By Corollary \ref{3samelength}).}
  \end{align*}

  \item
  Here symbolic Rees algebra is generated in degrees 1, $n+1$ and $m+1.$\\
 \begin{align*}
  I^{(s)} &= \displaystyle{\sum_{p+q(n+1)+r(m+1)=s}{I^p{(I^{(n+1)})}^q}{(I^{(m+1)})}^r}  \\
   &= \displaystyle{\sum_{\substack{t_1,t_2\\0\leq t_1\leq k_1~and \\0\leq t_2\leq k_2 ~and \\ s-t_1(n+1)-t_2(m+1)\geq 0}}{I^{s-t_1(n+1)-t_2(m+1)}(c_1)^{t_1}(D_{\min}(m+1))^{t_2}}} \text{ (By Theorem \ref{3differentlength}).}
\end{align*}

  \end{enumerate}
\end{proof}


\begin{proposition}\label{Dm}
Let $G$ be the clique sum of two odd cycles  $C_1=(x_1,\ldots,x_{2n+1})$ and $C_2=(x_1,y_2,\ldots,y_{2m+1})$ joined at a vertex $x_1$ with $n<m$.  Let $I=I(G)$ be the edge ideal of $G$ and $k= \Bigl\lfloor\dfrac{t}{n+1}\Bigr\rfloor .$ Then

\begin{equation*}
  \sdefect(I,t)\leq\begin{cases}
    \displaystyle{\sum_{i=1}^{k}{{t-i(n+1)+2n+2m+1}\choose{2n+2m+1}}}, & \text{for  $n+1\leq t\leq m $}.\\
    \displaystyle{\sum_{i=1}^{k}{{t-i(n+1)+2n+2m+1}\choose{2n+2m+1}}}+1, & \text{ for $t=m+1$}.
  \end{cases}
\end{equation*}

\end{proposition}
\begin{proof}
From Theorem \ref{3differentlength}, it follows that $I^{(t)}=I^t+(D_{\min}(t))$ for $n+1\leq t\leq m+1.$ Therefore to compute symbolic defect we need to count the minimal number of generators of $(D_{\min}(t)).$
 To find the cardinality of $D_{\min}(t),$ we need to find the number of solution of the equation $i+(n+1)s+(m+1)b=t$ for $i<t.$ For $b\geq 1,$ there is only one solution namely $(i,s,b)=(0,0,1)$ and it is for $t=m+1$.  So to find the other solutions we  assume $b=0$ and count the  number of solutions of the equation $i+(n+1)s=t$ with $i<t$. As $0\leq i< t,$ therefore $1\leq s\leq \Bigl\lfloor\dfrac{t}{n+1}\Bigr\rfloor.$ Thus the cardinality of the set $D_{\min}(t)$ is $\Bigl\lfloor\dfrac{t}{n+1}\Bigr\rfloor $ for $1\leq t\leq m$ and cardinality of $D_{\min}(m+1)$ is $\Bigl\lfloor\dfrac{t}{n+1}\Bigr\rfloor +1.$ The elements of $D_{\min}(t)$ are of the form $I^i(c_1)^s(c_2)^b.$ If $b=0$ then the elements $D_{\min}(t)$ are of the form $I^i(c_1)^s$ and for different $i,s$ there are elements be computed repeatedly. Therefore\\ $\sdefect(I,t)\leq \displaystyle{\sum_{i=1}^{k}{\mu(I^{t-i(n+1)})}}=\displaystyle{\sum_{i=1}^{k}{{t-i(n+1)+2n+2m+1}\choose{2n+2m+1}}}
 \text{ for  $n+1\leq t\leq m $ }$  and \\$\sdefect(I,t)\leq \displaystyle{\sum_{i=1}^{k}{\mu(I^{t-i(n+1)})}}+1 =\displaystyle{\sum_{i=1}^{k}{{t-i(n+1)+2n+2m+1}\choose{2n+2m+1}}}+1$ for $t=m+1.$
\end{proof}

 \section{Symbolic powers of edge ideals of complete graph}
 Throughout this section $G$ is a complete graph with $n$ vertices and $I=I(G)$ be the edge ideal of $G$. In this section we describe the generators of the symbolic powers of $I$ and calculate Waldschmidt constant, the resurgence and find the symbolic defect partially. We prove Minh's conjecture by showing that the regularity of symbolic powers and ordinary powers of $I$ are equal. We know that $I^{(t)}=(L(t))+(D(t))$. In \cite[Theorem 6.4]{janssen2017comparing} it is known that for complete graph $(L(t))=I^t$. Thus it is enough to understand the generators of $(D(t))$. The following lemma describes the generators of $(D(t))$.


 \begin{lem} Let $G$ be a complete graph with $n$ vertices and $I=I(G)$ is the edge ideal. Then for $t\geq 2$ we have

\begin{equation*}
\begin{split}
D(t) = \{x_1^{a_1}x_2^{a_2}\cdots x_n^{a_n}& \mid a_{i_1}+a_{i_2}+\cdots+a_{i_{n-1}}\geq t \mbox{ for } \{i_1,i_2,\ldots,i_{n-1}\}\subseteq\{1,2,\ldots,n\} \\
           &    \mbox{ and } a_1+\cdots +a_n\leq 2t-1\}
\end{split}
\end{equation*}
\end{lem}
\begin{proof}
  Since $G$ is a complete graph in $n$ vertices, any set of $(n-1)$ vertices forms a minimal vertex cover for $G.$ Let $x^{\underline{a}}={x_1^{a_1}x_2^{a_2}\cdots x_n^{a_n}\in D(t) }$ then $a_1+\dots +a_n\leq 2t-1$ and since $x^{\underline{a}}\in I^{(t)}$ for any minimal vertex cover $V,$ $W_V(x^{\underline{a}})\geq t$ which implies that  $a_{i_1}+a_{i_2}+\dots+a_{i_{n-1}}\geq t \mbox{ for } \{i_1,i_2,\dots,i_{n-1}\}\subseteq\{1,2,\dots,n\}$. Hence the result follows.
\end{proof}

Next theorem describes the generators of $\mathcal{R}_s(I)$. Since complete graph is perfect graph thus by \cite{flores}, $G$ is implosive graph. 
Hence by Theorem \ref{1implosive}, we need to  find all the induced indecomposable subgraph of $G$.

\begin{thm}\label{4complete}
For $1\leq s\leq n-1$ we get $I^{(s)}= I^s+(D(s)).$ And for any $s\geq n$ we have
$$I^{(s)}=\sum_{\substack{(r_1,\dots,r_{n-1})\\ s=r_1+2r_2+\cdots+ (n-1)r_{n-1}}}{I^{r_1}{{I^{(2)}}^{r_2}}\cdots{{I^{(n-1)}}^{r_{n-1}}}}$$
\end{thm}
\begin{proof}
   Here all induced subgraphs are induced indecomposable subgraphs. Therefore symbolic Rees algebra is minimally generated in degrees $1,2,\ldots,n-1$. Thus it is enough to find $D(s)$ for $2\leq s\leq n-1$. For $1\leq s\leq n-1$ we get $I^{(s)}= I^s+(D(s)),$ and if $s\geq n$ then $I^{(s)}$ is generated by $I,I^{(2)},\dots,I^{(n-1)}.$ Hence the result follows.
\end{proof}

Next we compute the Waldschmidt constant, resurgence  and symbolic defect for $G$.
\begin{thm}\label{4completeinvariant}
  Let $G$ be the complete graph on the vertices $\{x_1,\ldots,x_{n}\}$ and $I=I(G)$ be its edge ideal. Then
  \begin{enumerate}
    \item For any $ s\in \mathbb{N}$, $ \alpha{(I^{(s)})}$
     = $ s+ \Bigl\lceil\dfrac{s}{n-1}\Bigr\rceil.\qquad$\\
     Particularly, the Waldschmidt constant of $I$ is given by \begin{align*}\widehat{\alpha}{(I)} =\frac{n}{n-1}\end{align*}
    \item  $\alpha{(I^{(s)})} < \alpha{(I^t)}~~ if~~ and~~ only ~~if~~ I^{(s)}\nsubseteq I^t$
    \item  The resurgence of $I$ is given by $\rho(I) = \frac{2n-2}{n}$
  \end{enumerate}
\end{thm}

\begin{proof}
\begin{enumerate}
  \item Since $G$ is a complete graph, any minimal vertex cover of $G$ will be of the form $x_{i_1,}x_{i_2},\ldots ,x_{i_{n-1}} $ where $ \{i_1,i_2,\dots,i_{n-1}\}\subseteq\{1,2,\dots,n\}$. Note that for  $1\leq s\leq n-1,$ no monomial of degree $\leq s$ is in $I^{(s)},$ where as  $x_1x_2\dots x_{s+1}\in I^{(s)}.$ Therefore $\alpha(I^{(s)})=s+1$ for $1\leq s\leq n-1.$ From Theorem \ref{4complete}, it follows that for $s\geq n,$  $$\alpha(I^{(s)})=\min\{2r_1+3r_2+\dots+nr_{n-1}\mid s=r_1+2r_2+\dots+(n-1)r_{n-1}\}.$$ Now $2r_1+3r_2+\dots+nr_{n-1}=s+r_1+r_2+\dots+r_{n-1}.$ Then it is equivalent to find the minimum of ${r=r_1+r_2\dots+r_{n-1}}$ with the condition $s=r_1+2r_2+\dots+(n-1)r_{n-1}.$ Write $s=k(n-1)+p$ for some $k\in \mathbb{Z}$ and $0\leq p\leq n-2.$ Then observe that minimum value of $r$ will occur for maximum value of $r_{n-1}$ and the maximum value of $r_{n-1}$ is $k.$ Therefore the minimal generating degree term will come from $I^{(n-1)^k}I^{(p)}.$ Hence
\begin{align*}
                                    \alpha(I^{(s)}) & =  \left\{\begin{array}{ll}
                                     kn+(p+1)& \text{ if $p \neq 0$} \\
                                       kn & \text{  if $p=0$ }\\
                                        \end{array} \right.\\
                                         &= \left\{\begin{array}{ll}
                                        s+1+ \Bigl\lfloor\dfrac{s}{n-1}\Bigr\rfloor & \text{ if $p \neq 0$} \\
                                         & \\
                                         s+ \Bigl\lfloor\dfrac{s}{n-1}\Bigr\rfloor & \text{  if $p=0$ .}\\
                                         \end{array} \right.
              \end{align*}

     Therefore  $ \alpha{(I^{(s)})}$
     =$ s+ \Bigl\lceil\dfrac{s}{n-1}\Bigr\rceil.$
     Moreover $\frac{s}{n-1}\leq\Bigl\lceil\dfrac{s}{n-1}\Bigr\rceil\leq \frac{s}{n-1}+1.$ Thus \begin{align*}\widehat{\alpha}{(I)}=\lim_{s\rightarrow\infty}{} \frac{\alpha{(I^{(s)})}}{s} = 1+ \frac{1}{n-1}=\frac{n}{n-1}.\end{align*}

  \item Follows from \cite[Lemma 5.5]{janssen2017comparing}.
  \item Let $T =\tiny\{ \frac{s}{t} ~~|~~I^{(s)}\nsubseteq I^t \tiny\}. $
  For any $\frac{s}{t}\in T$ we have $\alpha(I^{(s)})<\alpha(I^t).$ By part (1) it follows that $ s+ \Bigl\lceil\dfrac{s}{n-1}\Bigr\rceil< 2t.$ This implies that $\frac{s}{t}<\frac{2(n-1)}{n}.$ So $\rho(I)\leq \frac{2(n-1)}{n}.$ By \cite[Theorem 1.2]{guardo2013}, we have ${\alpha(I)}/{\hat{\alpha}(I)}\leq \rho(I).$ Thus by part (1), this gives that $\frac{2(n-1)}{n}\leq \rho(I)$. Hence $\rho(I)=\frac{2(n-1)}{n}$.
\end{enumerate}

\end{proof}

\begin{thm}
Let $G$ be a complete graph with $n$ vertices and $I=I(G)$ be the edge ideal of $G$. Then for  $2\leq s\leq n-1$ we have $$\sdefect(I,s)={n \choose s+1}.$$

\end{thm}

\begin{proof}
 Since $G$ is a complete graph we have $I^{(s)}=I^s+(D(s))$. Therefore in order to compute $\sdefect(I,s),$ we count the number of minimal generators of the ideal $(D(s)).$ Then the following set give the minimal generators of the ideal $(D(s)),$
\begin{equation*}
\begin{split}
D_{\min}(s) = \{x_1^{a_1}x_2^{a_2}\cdots x_n^{a_n}& \mid a_{i_1}+a_{i_2}+\dots+a_{i_{n-1}}\geq s \mbox{ for } \{i_1,i_2,\dots,i_{n-1}\}\subseteq\{1,2,\dots,n\} \\
           &    \mbox{ and } a_1+\dots +a_n= \alpha(I^{(s)})\}.
\end{split}
\end{equation*}
Any other elements in $D(s)$ with condition $a_1+\dots+a_n>\alpha(I^{(s)})$ will be generated by the elements of $D_{\min}(s).$
 From Theorem \ref{4completeinvariant}, it follows that for $1\leq s\leq n-1,$ $\alpha(I^{(s)})=s+1.$ Therefore $a_i\in\{0,1\}.$ Then it is clear that the monomial $x_1^{a_1}x_2^{a_2}\cdots x_n^{a_n}$ will be an element of $D_{\min}(s)$ if and only if $(s+1)$ co-ordinates in that tuple will be $1$ and rest will be $0$. Hence number of solutions will be $n\choose {s+1}$.

\end{proof}

Now we show that Minh's conjecture is true for complete graphs. In order to prove this we need the following lemma.

\begin{lem}\label{artinian}
  Let $G$ be a complete graph with $n$ vertices and $I=I(G)$ is the corresponding edge ideal in the polynomial ring $R=k[x_1,\ldots,x_n]$. Let $\m=(x_1,\ldots,x_n)$ is the maximal homogeneous  ideal of $R$, then $\m^{t-1}I^{(t)}\subseteq I^t.$
\end{lem}

\begin{proof}

  Let $x^{\underline{a}}=x_1^{a_1}x_2^{a_2}\cdots x_n^{a_n}\in I^{(t)}$ and $x_1^{b_1}x_2^{b_2}\cdots x_n^{b_n}\in \m^{t-1}$, which implies $a_{i_1}+a_{i_2}+\dots+a_{i_{n-1}}\geq t \mbox{ for } \{i_1,i_2,\dots,i_{n-1}\}\subseteq\{1,2,\dots,n\}$ and $b_1+\dots+b_n\geq t-1.$ Then any monomial in $\m^{t-1}I^{(t)}$ will be of the form $x_1^{a_1+b_1}x_2^{a_2+b_2}\cdots x_n^{a_n+b_n}=x^{\underline a+\underline b}.$ Since $G$ is a complete graph, if any monomial  is written in the optimal form then there will be  at most one ancillary. Let $x_1$ is the ancillary with degree at least $2$. Let $e_{ij}$ denote the edge between $x_i$ and $x_j.$ Then in the optimal form only the edges of the form $e_{1i}$ will be present. Thus  $x^{\underline a+\underline b}= x_1^ce_{12}^{a_2+b_2}e_{13}^{a_3+b_3}\cdots e_{1n}^{a_n+b_n}$ where $c\geq 2$. Now $b(x_1^ce_{12}^{a_2+b_2}e_{13}^{a_3+b_3}\cdots e_{1n}^{a_n+b_n})= a_2+\dots+ a_n+b_2+\dots+b_n \geq t$, which implies $x^{\underline a+\underline b}\in I^t$. Hence $\m^{t-1}I^{(t)}\subseteq I^t.$
  Now  consider the case when $x_1$ is an ancillary  of degree at most $1. $
  We would like to prove that $\deg(x^{\underline{a}})\geq t+1.$
Let us assume that $\deg{(x^{\underline{a}})}< t+1.$ So $\deg(x^{\underline{a}})\leq t,$ but we also have that $a_1+\dots+a_{n-1}\geq t$ which implies that $a_n=0.$ Similarly we can show that $a_1=a_2=\cdots=a_n=0$ which is a  contradiction. Therefore $\deg(x^{\underline{a}})\geq t+1.$ Thus the degree of $x^{\underline a+\underline b}$ is at least $(t-1)+(t+1)=2t.$ Now in $x^{\underline a+\underline b}$ if the ancillary is of degree $0$ then nothing to prove. Let us assume that the ancillary is of degree $1$. Then  $x^{\underline a+\underline b}=x_1\prod_{i,j}e_{ij}^{b_{ij}}.$ which implies that  $\deg(\prod_{i,j}e_{ij}^{b_{ij}})\geq 2t-1$. Hence $\deg(\prod_{i,j}e_{ij}^{b_{ij}})\geq 2t.$ Thus $x^{\underline a+\underline b}\in I^t.$
\end{proof}

 \begin{thm}
 Let $G $ be a complete graph with $n$ vertices and $I=I(G)$ be the edge ideal of $G$. Then for any $s\geq 1,$ we have $$\reg I^{(s)}=\reg I^s.$$
 \end{thm}
 \begin{proof}
   By Lemma \ref{artinian} we have $\m^{s-1}I^{(s)}\subseteq I^s,$ so $I^{(s)}/I^s$ is an Artinian module. Therefore $\dim {I^{(s)}}/{I^s}=0$ and hence $H_{\m }^i(I^{(s)}/I^s)=0$ for $i>0.$ Consider the following short exact sequence
   $$0\rightarrow I^{(s)}/I^s \rightarrow R/I^s \rightarrow R/I^{(s)}\rightarrow 0. $$
   Applying local cohomology functor we get $H_{\m}^i(R/I^{(s)})\cong H_{\m}^i(R/I^{s})$ for $i\geq 1$ and the following short exact sequence
   \begin{equation} \label{exact}
0\rightarrow H_{\m}^0(I^{(s)}/I^s) \rightarrow H_{\m}^0(R/I^s) \rightarrow H_{\m}^0(R/I^{(s)})\rightarrow 0.
   \end{equation}

   Now from $(\ref{exact})$  it follows that $a_0(R/I^{(s)})\leq a_0(R/I^s).$
   So we can conclude that $\reg R/I^{(s)}=\max\{a_i(R/I^{(s)})+i~|~i\geq 0\}\leq\max\{a_i(R/I^{(s)})+i~|~i\geq 0\}=\reg(R/I^s). $ Therefore $\reg I^{(s)}\leq \reg I^s.$
   It follows from \cite[Theorem 4.6]{gu2018symbolic} that $\reg I^{(s)}\geq 2s.$ As complete graph is co-chordal graph, so by \cite{herzog2004} $\reg I^{s}= 2s.$ Thus we have $\reg I^{(s)}=\reg I^s.$

 \end{proof}


\begin{thebibliography}{AAAA}

\bibitem{bocci2016}
{C. Bocci, S. Cooper, E. Guardo, B. Harbourne, M. Janssen, U. Nagel, A. Seceleanu, A. Van Tuyl and T. Vu}, {The Waldschmidt constant for squarefree monomial ideals}, J. Algebraic Combin. {\bf 44}(2016), no. 4, 875--904.

\bibitem{bocci2010}
{C. Bocci and B. Harbourne}, {Comparing powers and symbolic powers of ideals}, J. Algebraic Geom. {\bf 19}(2010), no. 3,
399-417.


\bibitem{flores}
{A. Flores-M´endez, I. Gitler and E. Reyes}, {Implosive graphs: square-free monomials on symbolic Rees
algebras}, J. Algebra Appl. {\bf16} (2017), no. 8, 1750145 (23 pps).






\bibitem{ggsv}
{ F. Galetto, A. V. Geramita, Y-S. Shin and A. Van Tuyl}, {The symbolic defect of an ideal}, Journal of Pure and Applied Algebra, {\bf223} (2019), no. 6, {2709--2731}.
\bibitem{gu2018symbolic}
{ Y. Gu, H. T. Ha, J. L. O'Rourke and J. W. Skelton}, {Symbolic powers of edge ideals of graphs}, Preprint (2018), arXiv:1805.03428.

\bibitem{guardo2013}
{E. Guardo, B. Harbourne and A. Van Tuyl}{ Asymptotic resurgences for ideals of positive dimensional subschemes of projective space}, Adv. Math. {\bf 246}(2013), 114--127.

\bibitem{Frank}
{F. Harary and M. D. Plummer,} {On indecomposable graphs}, Canad. J. Math. {\bf 19} (1967), 800--809.
\bibitem{herzog2004}
 {J. Herzog, T. Hibi and X. Zheng}, {Monomial ideals whose powers have a linear resolution}, Math. Scand. {\bf 95} (2004), no. 1, 23-32.

\bibitem{janssen2017comparing}
{M. Janssen, T. Kamp and J. Vander Woude}, {Comparing Powers of Edge Ideals}, Preprint (2017), arXiv:1709.08701.

\bibitem{jayanthan}
{A. V. Jayanthan and R. Kumar}, {Regularity of Symbolic Powers of Edge Ideals}, Preprint (2019), arXiv:1903.05313.

\bibitem{lam2015associated}
{ H.M. Lam and N.V. Trung},{ Associated primes of powers of edge ideals and ear decompositions of graphs}, { To appear in Transactions of the American Mathematical Society}.


\bibitem{bernal}
{J.Martnez-Bernal, C. Rentera and R. H. Villarreal}, {Combinatorics of symbolic Rees algebras of edge ideals of clutters}, Commutative algebra and its connections to geometry, Contemp. Math. {\bf 555}(2011) 151--164.

\bibitem{fakhari}
{S. A. Seyed Fakhari}, {Regularity of symbolic powers of edge ideals of Unicyclic graphs}, Preprint (2019), arXiv:1903.10962v1.



\bibitem{svv}
{A. Simis, W. Vasconcelos and R. H. Villarreal}, {On the ideal theory of graphs}, J. Algebra {\bf 167} (1994), 389--416.

\bibitem{vantuylbook}
{A.Van Tuyl}, {A beginner's guide to edge and cover ideals. Monomial ideals}, computations and applications, 63--94, Lecture Notes in Math., 2083, Springer, Heidelberg, 2013.




 \end{thebibliography}
 \end{document}